\documentclass[10pt]{amsart}
\usepackage{amssymb}
\usepackage{bm}
\usepackage{graphicx}
\usepackage[centertags]{amsmath}
\usepackage{amsfonts}
\usepackage{amsthm}
\usepackage{graphicx}
\newtheorem{thm}{Theorem}
\newtheorem{cor}[thm]{Corollary}
\newtheorem{lem}[thm]{Lemma}

\newtheorem{prop}[thm]{Proposition}

\newtheorem{fact}[thm]{Fact}

\newtheorem{defn}[thm]{Definition}
\theoremstyle{definition}

\newcommand{\ee}{\varepsilon}

\newcommand{\meg}{\geqslant}
\newcommand{\mik}{\leqslant}

\newcommand{\con}{\smallfrown}

\makeindex

\begin{document}

\title{Oscilation stability for continuous monotone surjections}
\author{Stevo Todorcevic}
\author{Konstantinos Tyros}

\address{Department of Mathematics, University of Toronto, Toronto, Canada, M5S 2E4. Institut de Math\'ematiques de Jussieu, UMR 7586, 2 pl. Jussieu, case 7012, 75251 Paris Cedex 05, France. }
\email{stevo@math.toronto.edu. stevo@math.jussieu.fr}
\address{and}
\address{Department of Mathematics, University of Toronto, Toronto, Canada, M5S 2E4 }
\email{ktyros@math.toronto.edu}

\thanks{\textit{Key words}: Dual Ramsey Theory, Ramsey Degree, Cantor set.}

\thanks{Research supported by grants from NSERC and CNRS.}

\begin{abstract}
 We prove that for every integer $b\geqslant 2$ and positive real $\varepsilon$ there exists a finite number $t$ such that for every finite coloring of the nondecreasing surjections from $b^\omega$ onto $b^\omega$, there exist $t$ many colors such that their $\varepsilon$-fattening contains a cube.
\end{abstract}
\maketitle

\section{Introduction}

Recall the statement of the dual Ramsey theorem for infinite partitions of $\omega$ (see \cite{CS} or \cite{To}): For every finite Borel coloring of the space $C^r_{\text{sur}}(\omega)$
of all rigit surjections from $\omega$ onto $\omega$ there is a rigit surjection $h:\omega\rightarrow\omega$ such that the set
$$C^r_{\text{sur}}(\omega)\upharpoonright h=\{f\circ h: f\in C^r_{\text{sur}}(\omega)\}$$
is monochromatic. In this note we examine this kind of dual Ramsey statement with the index-set $\omega$ replaced by the Cantor set $2^\omega,$ or more generally,
powers of the form $b^\omega$ for $b$ any positive integer. More precisely, we focus on the space $C^\uparrow_{\text{sur}}(b^\omega)$ of all nondecreasing surjections and we examine to which extend a similar result holds. Unlike the dual Ramsey theorem, in our case the structure under consideration admits a Ramsey degree and this degree can be realized only in an approximate sense. In Section \ref{Aprox_sec} we establish the necessity of the approximations, while in Section \ref{deg_sec} we prove that the Ramsey degree provided by the statement of the main result (Theorem \ref{main_thm} below) is the best possible. However, to state our result precisely we need some notation.

 By $\omega$, we denote the set of the natural numbers starting from 0. For every $k$ in $\omega$, $k$ also stands for the set of the natural numbers strictly less that $k$. For $b,k\in\omega$, by $b^k$ (resp. $b^{<k}$) we denote the set of all sequences in $b$ of length $k$ (resp. strictly less than $k$) and by $b^\omega$ (resp. $b^{<\omega}$) we denote the set of all sequences in $b$ of infinite (resp. finite) length.
For $2\mik b<\omega$, it is well known that the space $b^\omega$ is a metrizable compact space. Throughout this note we will consider the following metric witnessing this fact. For every distinct $x,y$ in $b^\omega$, we set $\rho_{b}(x,y)=2^{-n_0}$ where $n_0=\min\{n<\omega:x(n)\neq y(n)\}$. Moreover we endow the set $b^\omega$ with the lexicographical order $\leq_\text{lex}$, i.e. for $x,y\in b^\omega$, we write $x\leq_\text{lex}y$ if either $x=y$ or $x(n_0)<y(n_0)$ where $n_0=\min\{n<\omega:x(n)\neq y(n)\}$. Then $(b^\omega,\leq_\text{lex})$ is a linearly ordered set. Similarly the lexicographical order $\leq_\text{lex}$ is defined on $\sqsubseteq$-incomparable pairs of $b^{<\omega}$ inducing a linear order on every subset of $b^{<\omega}$ consisting of pairwise $\sqsubseteq$-incomparable elements.

We are interested in the following subspace of the continuous maps from $b^\omega$ into itself
\[C^\uparrow_{\text{sur}}(b^\omega)=\{f:b^\omega\to b^\omega:f\;\text{is continuous, onto and nondecreasing}\},\]
where by nondecreasing we mean $f(x)\leq_\text{lex} f(y)$ whenever $x\leq_\text{lex} y$.
We endow $C^\uparrow_{\text{sur}}(b^\omega)$ with the following metric
\[\rho_\infty(f,g)=\sup\{\rho_b(f(x),g(x)):x\in b^\omega\}\]
for all $f,g\in C^\uparrow_{\text{sur}}(b^\omega)$.
Finally, let us recall the sequence of the odd tangent numbers $(t_k)_{k=1}^\infty$ defined by $t_k=\tan^{2k-1}(0)$ for every positive integer $k$.
The main result of this note is the following.

\begin{thm}
\label{main_thm}
  Let $b\in\omega$ with $b\meg2$. Then for every positive real $\ee$ there exists a positive integer $t=t(\ee)$ such that for every positive integer $K$ and every coloring $c:C^\uparrow_{\text{sur}}(b^\omega)\to K$ there exist $h\in C^\uparrow_{\text{sur}}(b^\omega)$ and $B\subseteq K$ with at most $t$ elements such that for every $f\in C^\uparrow_{\text{sur}}(b^\omega)$ there exists $g\in C^\uparrow_{\text{sur}}(b^\omega)$ satisfying $\rho_\infty(g,f\circ h)<\ee$ and $c(g)\in B$. In particular, $t=t(\ee)=t_{b^k-1}$ with $k=\lfloor\log_2(1/\ee)\rfloor+1$.
\end{thm}


To state a corollary of this result we need the following notion whose relationship to Ramsey theory was already pointed out before (see \cite{KPT}). Let $(X,d)$ be a metric space, $t$ a positive integer and $\delta$ a positive real. We will say that a subset $Y$ of $X$ is of $\delta$-covering number at most $t$, if there exists a finite subset $A$ of $X$ of cardinality at most $t$ such that $Y\subseteq \cup_{x\in A}B_d(x,\delta)$. Under this terminology, the above result has the following immediate consequence.

\begin{cor}
  Let $b\in\omega$ with $b\meg2$. Then for every positive real $\ee$ there exists a positive integer $t$ such that for every compact metric space $(K,d)$, every map $c:C^\uparrow_{\text{sur}}(b^\omega)\to K$ and every positive real $\delta$ there exist $h\in C^\uparrow_{\text{sur}}(b^\omega)$ and $B\subseteq K$ of $\delta$-covering number $t$ such that for every $f\in C^\uparrow_{\text{sur}}(b^\omega)$ there exists $g\in C^\uparrow_{\text{sur}}(b^\omega)$ satisfying $\rho_\infty(g,f\circ h)<\ee$ and $c(g)\in B$.
\end{cor}

Moreover Theorem \ref{main_thm} has the following corollary.

\begin{cor}
  \label{cor1}
  Let $b\in\omega$ with $b\meg2$. Then for every positive reals $\ee,M$ there exists a positive integer $t=t(\ee,M)$ such that for every bounded metric space $(K,d)$, every $M$-Lipschitz map $c:C^\uparrow_{\text{sur}}(b^\omega)\to K$ there exist $h\in C^\uparrow_{\text{sur}}(b^\omega)$ and $B\subseteq K$ of $\ee$-covering number at most $t$ such that for every $f\in C^\uparrow_{\text{sur}}(b^\omega)$ we have $c(f\circ h)\in B$. In particular, $t(\ee,M)=t(\frac{\ee}{2M})$.
\end{cor}
\begin{proof}
  Let $\ee,M$ be positive reals and $t=t(\frac{\ee}{2M})$. Also let $(K,d)$ be a bounded metric space and $c:C^\uparrow_{\text{sur}}(b^\omega)\to K$ be an $M$-Lipschitz map. By the boundness of $(K,d)$ there exist a positive integer $N$ and $x_0,\ldots,x_{N-1}\in K$ such that $\cup_{i=0}^{N-1} B_d(x_i,\ee/2)=K$, where $B_d(x,\delta)=\{y\in X:d(x,y)<\delta\}$. We define $\widetilde{c}:C^\uparrow_{\text{sur}}(b^\omega)\to N$ by setting $\widetilde{c}(f)=\min\{i\in N:c(f)\in B(x_i,\ee/2)\}$ for all $f\in C^\uparrow_{\text{sur}}(b^\omega)$. By the choice of $t$ there exist $h\in C^\uparrow_{\text{sur}}(b^\omega)$ and $F\subset N$ of cardinality at most $t$ such that for every $f\in C^\uparrow_{\text{sur}}(b^\omega)$ there exists $g_f\in C^\uparrow_{\text{sur}}(b^\omega)$ satisfying
  \begin{enumerate}
    \item[(i)] $\rho_\infty(f\circ h,g_f)<\frac{\ee}{2M}$ and
    \item[(ii)] $\widetilde{c}(g_f)\in F$.
  \end{enumerate}
  We set $B'=\bigcup_{i\in F} B_d(x_i,\ee/2)$ and $B=\bigcup_{i\in F} B_d(x_i,\ee)$. Clearly $B$ is of $\ee$-covering number $t$. By (i) and the fact that $c$ is $M$-Lipschitz , for every $f\in C^\uparrow_{\text{sur}}(b^\omega)$, we have that $d(c(g_f),c(f\circ h))<\ee/2$. By (ii) above and the definition of $\widetilde{c}$ we have that $c(g_f)\in B'$, for all $f\in C^\uparrow_{\text{sur}}(b^\omega)$. Thus $c(f\circ h)\in B$, for all $f\in C^\uparrow_{\text{sur}}(b^\omega)$.
\end{proof}



\section{Clopen interval partitions}\label{filt_sec}
Let us start with the following fact concerning the minimum and maximum elements of nonempty closed or clopen subsets of $b^\omega$ with respect to $\leq_\text{lex}$.
\begin{fact}
  \label{finding_max}
  Let $b<\omega$ with $b\meg 2$. Every nonempty closed subset of $b^\omega$ admits a maximum and a minimum. Moreover if the subset is clopen then its maximum is eventually equal to $b-1$ and its minimum is eventually equal to $0$.
\end{fact}
\begin{proof}
  Let $F$ be a nonempty closed subset of $b^\omega$. Clearly $F$ is a compact set. For every $n<\omega$, we pick a finite subset $G_n$ of $F$ such that $F\subseteq\cup_{x\in G_n}B_{\rho_b}(x,1/n)$ and we set $x_n$ to be the maximum of $G_n$ with respect to $\leq_\text{lex}$. It is easy to check that $x_m\in B_{\rho_b}(x_n,1/n)$, for all $n\mik m<\omega$. Hence $(x_n)_n$ is a Cauchy sequence and therefore converges to some element $x$ of $F$. Moreover we have that
  \begin{equation}
  \label{e01}
    \rho_b(x,x_n)\mik1/n
  \end{equation} for all $n<\omega$. We claim that $x$ is the maximum of $F$. Indeed, assume on the contrary that there exists $y\in F$ such that $x<_\text{lex}y$.
  Let $n_0$ positive integer such that $1/n_0<\rho_b(x,y)/2$. Thus by (\ref{e01}) we have that $y\not\in B_{\rho_b}(x_{n_0},1/n_0)$ and since $x<_\text{lex} y$ we have that $z<_\text{lex}y$ for all $z\in B_{\rho_b}(x_{n_0},1/n_0)$. This in particular, by the choice of $x_{n_0}$, yields that $y\not\in\cup_{x\in G_{n_0}}B_{\rho_b}(x,1/n_0)$, which contradicts that the latter union covers $F$. Similar arguments yield that $F$ admits a minimum.

  Let us now assume that $F$ is a nonempty clopen subset of $b^\omega$. Assume that the maximum $x$ of $F$ is not eventually equal to $b-1$. Then we can pick a sequence $(x_n)_n$ in $b^\omega$ convergent to $x$ such that $x<_\text{lex}x_n$ for all $n<\omega$. Thus, since $x$ is the maximum of $F$ we have that $x_n$ belongs to the complement of $F$ for all $n<\omega$, which contradicts that $F$ is also open. Similar arguments yield that the minimum of $F$ is eventually equal to 0.
\end{proof}

Let $b<\omega$ with $b\meg2$. A subset $U$ of $b^\omega$ is called \emph{interval} if for every $x,y\in U$ and $z\in b^\omega$ satisfying $x\leq_\text{lex}z\leq_\text{lex}y$ we have that $z\in U$. Central role in our analysis possesses the following notion.
\begin{defn}\label{filtering_defn}
  Let $b<\omega$ with $b\meg2$. A family $(U_s)_{s\in b^{<\omega}}$ of nonempty clopen intervals of $b^\omega$ is called a \emph{filtering} on $b^\omega$, if the following are satisfied
  \begin{enumerate}
     \item[(i)] $U_\emptyset=b^\omega$,
     \item[(ii)] $(U_{s^\con(i)})_{i<b}$ is a (disjoint) partition of $U_s$ for every $s\in b^{<\omega}$ and
     \item[(iii)] $(\max U_{s^\con(i)})_{i<b}$ is $\leq_\text{lex}$-increasing for every $s\in b^{<\omega}$.
   \end{enumerate}
\end{defn}
 Let $\mathcal{F}_b$ be the set of all filterings on $b^\omega$. For $(V_s)_{s\in b^{<\omega}}$ and $(U_s)_{s\in b^{<\omega}}$ in $\mathcal{F}_b$ we will write $(V_s)_{s\in b^{<\omega}}\preccurlyeq(U_s)_{s\in b^{<\omega}}$ if for every $s\in b^{<\omega}$, the set $V_s$ is element of the algebra generated by the members of the family $(U_s)_{s\in b^{<\omega}}$. For $(U_s)_{s\in b^{<\omega}}$ in $\mathcal{F}_b$ we set
\[\mathcal{F}_b((U_s)_{s\in b^{<\omega}})=\{(V_s)_{s\in b^{<\omega}}\in\mathcal{F}_b:(V_s)_{s\in b^{<\omega}}\preccurlyeq(U_s)_{s\in b^{<\omega}}\}.\]
Moreover, for every $f\in C^\uparrow_{\text{sur}}(b^\omega)$, we set $(U^f_s)_{s\in b^{<\omega}}=(f^{-1}(W_s))_{s\in b^{<\omega}}$, where $W_s=\{x\in b^\omega:s\sqsubset x\}$ for all $s\in b^{<\omega}$. It is easy to check that $(U^f_s)_{s\in b^{<\omega}}$ is a filtering on $b^\omega$. The following lemma describes the relation between the elements of $C^\uparrow_{\text{sur}}(b^\omega)$ and the filterings on $b^\omega$. Finally, for $y\in b^\omega$ and $n<\omega$ by $x|n$ we denote the initial segment of $x$ of length $n$.

\begin{lem}
  \label{lem_indentification}
  Let $b<\omega$ with $b\meg2$. Then the map $D:C^\uparrow_{\text{sur}}(b^\omega)\to \mathcal{F}_b$ sending each $f$ to $(U^f_s)_{s\in b^{<\omega}}$ is 1-1 and onto. Moreover, for every $h\in C^\uparrow_{\text{sur}}(b^\omega)$ we have $\mathcal{F}_b((U^h_s)_{s\in b^{<\omega}})=\{(U^{f\circ h}_s)_{s\in b^{<\omega}}:f\in C^\uparrow_{\text{sur}}(b^\omega)\}$.
\end{lem}
\begin{proof}
  To prove that $D$ is 1-1, we fix $f\neq g$ in $C^\uparrow_{\text{sur}}(b^\omega)$. Then there exists $x\in b^\omega$ such that $f(x)\neq g(x)$. Pick $s\in b^{<\omega}$ such that $s$ is initial segment of $f(x)$ but not of $g(x)$. Thus $f(x)\in W_s$ and $g(x)\not\in W_s$. That is $x\in f^{-1}(W_s)=U^f_s$ and $x\not\in g^{-1}(W_s)=U^g_s$. Hence $U^f_s\neq U^g_s$ and therefore $D(f)\neq D(g)$.

  In order to prove that $D$ is onto, let us fix some $(U_s)_{s\in b^{<\omega}}$ in $\mathcal{F}_f$ and define $f:b^\omega\to b^\omega$ as follows. Fix an $x\in b^\omega$. Since $(U_s)_{s\in b^{<\omega}}$ is a filtering, by (i) and (ii) of definition \ref{filtering_defn} there exists a sequence $(s_n)_{n<\omega}$ in $b^{<\omega}$ such that $x\in U_{s_n}$, $s_n$ is of length $n$ and $s_n\sqsubset s_{n+1}$ for all $n<\omega$. Actually there exists unique such a sequence. Let $y$ be the unique element of $b^\omega$ satisfying $s_n\sqsubset y$ for all $n<\omega$. Set $f(x)=y$. Let us check that $f$ belongs to $C^\uparrow_{\text{sur}}(b^\omega)$ satisfying $(U^f_s)_{s\in b^{<\omega}}=(U_s)_{s\in b^{<\omega}}$. By (iii) of Definition \ref{filtering_defn} it follows that $f$ is increasing. To check that $f$ is onto let us fix some $y\in b^\omega$. Observe that $\cap_{n<\omega}U_{y|n}$ is non-empty. Picking any $x$ from this intersection we have that $f(x)=y$. To justify the continuity of $f$ let us fix a convergent sequence $(x_n)_{n<\omega}$ to so some $x$ in $b^\omega$. Let $y=f(x)$ and $y_n=f(x_n)$ for all $n<\omega$. Moreover we set $s_n$ to be the initial segment of $y$ of length $n$ for all $n<\omega$. By the definition of $f$ we have that $x\in U_{s_n}$. We pass to a subsequence $(x_{k_n})_{n<\omega}$ of $(x_n)_{n<\omega}$ such that $x_{k_n}\in U_{s_n}$ for every $n<\omega$. By the definition of $f$ we have that $s_n$ is initial segment of $y_{k_n}$ and by the definition of $(s_n)_{n<\omega}$ we get that $y_{k_n}$ converges to $y$. Hence $f$ is continuous. Up to now we have proven that $f$ belongs to $C^\uparrow_{\text{sur}}(b^\omega)$. To see that $(U^f_s)_{s\in b^{<\omega}}=(U_s)_{s\in b^{<\omega}}$ let us fix an arbitrary $s\in b^{<\omega}$. Then
  \[\begin{split}&x\in U_s \\
    \text{iff }&s\text{ is initial segment of }f(x)\\
    \text{iff }&f(x)\in W_s \\
    \text{iff }&x\in U^f_x.\end{split}\]

  To prove the second part of the lemma we fix some $h\in C^\uparrow_{\text{sur}}(b^\omega)$. Since every clopen set can be written as a finite union of the basic clopen sets $(W_s)_{s\in b^{<\omega}}$, we easily get that $\mathcal{F}_b((U^h_s)_{s\in b^{<\omega}})\supseteq\{(U^{f\circ h}_s)_{s\in b^{<\omega}}:f\in C^\uparrow_{\text{sur}}(b^\omega)\}$. In order to prove the inverse inclusion, let us pick $(V_s)_{s\in b^{<\omega}}$ from $\mathcal{F}_b((U^h_s)_{s\in b^{<\omega}})$. Since $D$ is onto, there exists $g\in C^\uparrow_{\text{sur}}(b^\omega)$ such that $(U^g_s)_{s\in b^{<\omega}}=(V_s)_{s\in b^{<\omega}}$. It suffices to find $f\in C^\uparrow_{\text{sur}}(b^\omega)$ such that $g=f\circ h$. We define $f$ as follows. Let $y\in b^\omega$ and $A_y=\cap_{n<\omega}U^h_{y|n}$. Since $(U^g_s)_{s\in b^{<\omega}}\in\mathcal{F}_b((U^h_s)_{s\in b^{<\omega}})$, we have that there exists a sequence $(s_n)_{n<\omega}$ such that $A_y\subseteq U^g_{s_n}$, $s_n$ is initial segment of $s_{n+1}$ and $s_n$ is of length $n$ for all $n<\omega$. Finally, let $z$ be the unique element in $b^\omega$ such that $s_n$ is initial segment of $z$ for all $n<\omega$ and set $f(y)=z$. Arguing as in the proof of the first part of the lemma, one can show that $f$ belongs to $C^\uparrow_{\text{sur}}(b^\omega)$.  To prove that $f\circ h=g$ it suffices to show that $(U_s^{f\circ h})_{s\in b^{<\omega}}=(U_s^g)_{s\in b^{<\omega}}$. By the definition of the map $g$, it suffices to show that $(U_s^{f\circ h})_{s\in b^{<\omega}}=(V_s)_{s\in b^{<\omega}}$. First, let us observe that since $(V_s)_{s\in b^{<\omega}}\preccurlyeq(U_s^{h})_{s\in b^{<\omega}}$, we have that $x\in V_s$ if and only if $h^{-1}(\{h(x)\})\subseteq V_s$, for all $x\in b^\omega$ and $s \in b^{<\omega}$. Moreover, for every $x\in b^\omega$ we have that $A_{h(x)}=\cap_{n<\omega}U^h_{h(x)|n}=h^{-1}(\cap_{n<\omega} W_{h(x)|n})=h^{-1}(\{h(x)\})$. Hence for every $x\in b^\omega$ and $s \in b^{<\omega}$ we have
  \[\begin{split}
    & x\in V_s\\
    \text{iff }& h^{-1}(\{h(x)\})\subseteq V_s \\
    \text{iff }&A_{h(x)}\subseteq V_s\\
    \text{iff }&s\text{ is initial segment of } f(h(x)) \\
    \text{iff }&f(h(x)) \in W_s \\
    \text{iff }&x\in U_s^{f\circ h}.
  \end{split}\]
\end{proof}

\section{Order isomorphic copies of $\mathbb{Q}$ in $b^\omega$}\label{Devl_sec}
Let $b<\omega$ with $b\meg2$. We set $\mathcal{A}_b$ to be the set of the elements of $b^\omega$ being eventually equal to $b-1$ excluding $\max b^\omega$. It is easy to check that $(\mathcal{A}_b,\leq_\text{lex})$ is a countable unbounded dense linearly ordered set and therefore order isomorphic to $\mathbb{Q}$. Moreover, for every filtering $(U_s)_{s\in b^{<\omega}}$ on $b^\omega$, the subset $$\{\max U_s:s\in b^{<\omega}\}\setminus\{\max b^\omega\}$$ of $\mathcal{A}_b$ (see Fact \ref{finding_max}) is order isomorphic to $\mathbb{Q}$.

We shall need a result due to D. Devlin (see \cite{D} and \cite{To}). In order to state it we need some additional notation. For a linear ordered set $(P,\leq)$ and a positive integer $k$ by $[P]^k$ we denote the set of all $\leq$-increasing $k$-tuples in $P$. Moreover, let us recall the sequence of the odd tangent numbers $(t_k)_{k=1}^\infty$ defined by $t_k=\tan^{2k-1}(0)$ for every positive integer $k$.
\begin{thm}
  [D. Devlin]
  \label{Devlin_Thm}
  For every positive integer $l$ and every finite coloring of $[\mathbb{Q}]^l$, there exists a subset $Y$ of $\mathbb{Q}$ order isomorphic to $\mathbb{Q}$ such that $[Y]^l$ uses at most $t_l$ colors.
\end{thm}
The above has the following immediate consequence.
\begin{cor}
  \label{Devlin_Cor}
  Let $b<\omega$ with $b\meg2$. For every positive integer $k$ and every finite coloring of $[\mathcal{A}_b]^k$, there exists a subset $Y$ of $\mathcal{A}_b$ order isomorphic to $\mathcal{A}_b$ such that $[Y]^k$ uses at most $t_k$ colors.
\end{cor}
For every positive integer $k$, let $l_k=b^k-1$ and $(s_i^k)_{i=0}^{l_k}$ be the $\leq_\text{lex}$-increasing enumeration of $b^k$. Clearly for every positive integer $k$ and every $f\in C^\uparrow_{\text{sur}}(b^\omega)$ we have that $(\max U^f_{s^k_i})_{i<l_k}\in[\mathcal{A}_b]^{l_k}$.

\begin{lem}
  \label{finding_function}
  Let $b<\omega$ with $b\meg2$ and $k$ be a positive integer. For every $\mathbf{x}\in[\mathcal{A}_b]^{l_k}$ we have that there exists $f\in C^\uparrow_{\text{sur}}(b^\omega)$ such that $\mathbf{x}=(\max U^f_{s^k_i})_{i<l_k}$.
\end{lem}
\begin{proof}
  Let $\mathbf{x}=(x_i)_{i<l_k}\in[\mathcal{A}_b]^{l_k}$ and $(U_s)_{s\in b^k}$ be the unique partition of $b^\omega$ into nonempty clopen intervals satisfying $\max U_{s^k_i}=x_i$ for all $i<l_k$ and $\max U_{s^k_{l_k}}=\max b^\omega$. Pick any filtering $(V_s)_{s\in b^{<\omega}}$ such that $V_{s_i^{l_k}}=U_{s_i^{l_k}}$ for all $i\mik l_k$. Finally, by Lemma \ref{lem_indentification}, there exists $f\in C^\uparrow_{\text{sur}}(b^\omega)$ such that $(U^f_s)_{s\in b^{<\omega}}=(V_s)_{s\in b^{<\omega}}$. Clearly $f$ is the desirable one and the proof is complete.
\end{proof}
We will also need some notation concerning the order isomorphic copies of $\mathbb{Q}$ in $\mathcal{A}_b$. Let $b<\omega$ with $b\meg2$. Also let $Y$ be a subset of $\mathcal{A}_b$ order isomorphic to $\mathbb{Q}$. We set
\[[Y]^\eta=\{Z\subseteq Y: Z \;\text{is order isomorphic to}\;\mathbb{Q}\}.\]
Moreover, for every $f\in C^\uparrow_{\text{sur}}(b^\omega)$ let us set
  \[Y_f=\{\max U^f_s:s\in b^{<\omega}\}\setminus\{\min b^\omega\}.\]
It is easy to check that $Y_f\in [\mathcal{A}_b]^\eta$, for all $f\in C^\uparrow_{\text{sur}}(b^\omega)$. The relation between the elements of $C^\uparrow_{\text{sur}}(b^\omega)$ and $\mathcal{A}_b$ is even stronger and it is described by the following lemma.

\begin{lem}
  \label{lem_copies_Q_functions}
  Let $b<\omega$ with $b\meg2$. Then for every $Y\in[\mathcal{A}_b]^\eta$ there exists $f \in C^\uparrow_{\text{sur}}(b^\omega)$ such that $Y_f=Y$. More generally, for every $h\in C^\uparrow_{\text{sur}}(b^\omega)$ we have that $[Y_h]^\eta=\{Y_{f\circ h}:f\in C^\uparrow_{\text{sur}}(b^\omega)\}$.
\end{lem}
\begin{proof}
  Let $Y\in [\mathcal{A}_b]^\eta$. In order to determine $f\in C^\uparrow_{\text{sur}}(b^\omega)$ such that $Y_f=Y$, by Lemma \ref{lem_indentification}, it suffices to construct a filtering $(U_s)_{s\in b^{<\omega}}$ on $b^\omega$ such that $\{\max U_s:s\in b^{<\omega}\}\setminus\{\max b^\omega\}= Y$. Since $Y$ is countable, let $\{y_n:n<\omega\}$ be an enumeration of $Y$. We set $U_\emptyset=b^\omega$. Suppose that for some $k<\omega$ with $k>0$ the elements $(U_s)_{s\in b^{<k}}$ have been constructed. We are going to construct $(U_s)_{s\in b^k}$. Let $s\in b^{k-1}$. We set $i_0=\min\{i<\omega: y_i\in U_s\setminus\{\max U_s\}\}$ and
  for every $p<b-1$ with $p>0$ we inductively define $i_p=\min\{i<\omega: y_{i_{p-1}}<_\text{lex}y_i<_\text{lex}\max U_s\}$. We set $U_{s^\con (0)}=\{x\in U_s:x\leq_\text{lex} y_{i_0}\}$, for every $p<b-1$ with $p>0$ we set $U_{s^\con (p)}=\{x\in U_s:y_{i_{p-1}}<_\text{lex}x\leq_\text{lex} y_{i_p}\}$ and $U_{s^\con (b-1)}=\{x\in U_s:y_{i_{b-2}}<_\text{lex}x\}$.

  It is clear that $(U_s)_{s\in b^{<\omega}}$ is a filtering on $b^\omega$ and $\{\max U_s:s\in b^{<\omega}\}\setminus \{\max b^\omega\}\subseteq Y$. The inverse inclusion can be easily checked by showing using induction on $k$ that $\{y_i:i<k\}\subseteq \{\max U_s:s\in b^k\}$. The proof of the first part of the lemma is complete.

  Let $h\in C^\uparrow_{\text{sur}}(b^\omega)$. Also let $Y\in[Y_h]^\eta$. By the first part of the lemma we may pick $g\in C^\uparrow_{\text{sur}}(b^\omega)$ satisfying $Y_g=Y$. Since $Y_g\in[Y_h]^\eta$, we have that $(U_s^g)_{s\in b^{<\omega}}\preccurlyeq (U_s^h)_{s\in b^{<\omega}}$. Indeed, let $k<\omega$. Then $\{\max U^g_{s^k_i}:i\mik l_k\}\subseteq Y_g\cup\{\max b^\omega\}\subseteq Y_h\cup\{\max b^\omega\}$ (see the notation introduced before Lemma \ref{finding_function}) and therefore we may pick $k'<\omega$ such that $\{\max U^g_{s^k_i}:i\mik l_k\}\subseteq \{\max U^h_{s^{k'}_j}:j\mik l_{k'}\}$. Thus, there exist $0\mik j_0<j_1<\ldots<j_{l_k}=l_{k'}$ such that $\max U^g_{s^k_i}=\max U^h_{s^{k'}_{j_i}}$ for all $i\mik l_k$. Let $I_0=\{s^{k'}_{j}:j\mik j_0\}$, for every $i\mik l_k$ with $i>0$ let $I_i=\{s^{k'}_{j}:j_{i-1}<j\mik j_i\}$ and observe that $U^g_{s^k_i}=\cup_{s\in I_i}U^h_s$, for all $i\mik l_k$. Hence $(U_s^g)_{s\in b^{<\omega}}\preccurlyeq (U_s^h)_{s\in b^{<\omega}}$. By Lemma \ref{lem_indentification} there exists $f\in C^\uparrow_{\text{sur}}(b^\omega)$ such that $g=f\circ h$. Hence $Y=Y_g=Y_{f\circ h}$ and therefore $[Y_h]^\eta\subseteq\{Y_{f\circ h}:f\in C^\uparrow_{\text{sur}}(b^\omega)\}$. The inverse inclusion is straightforward since for every $f\in C^\uparrow_{\text{sur}}(b^\omega)$ we have that $(U_s^{f\circ h})_{s\in b^{<\omega}}\preccurlyeq (U_s^h)_{s\in b^{<\omega}}$ and therefore $Y_{f\circ h}\in [Y_h]^\eta$.
\end{proof}

\section{Proof of Theorem \ref{main_thm}}


 By the definition of the metric $\rho_\infty$ on $C^\uparrow_{\text{sur}}(b^\omega)$ the following is immediate.

\begin{lem}
  \label{being_close}
  Let $b<\omega$ with $b\meg2$ and $\ee$ be a positive real. For every $f,g\in C^\uparrow_{\text{sur}}(b^\omega)$ we have that $\rho_\infty(f,g)<\ee$ if and only if $(\max U^f_s)_{s\in b^{ k}}=(\max U^g_s)_{s\in b^{k}}$, where $k=\lfloor\log_2(1/\ee)\rfloor+1$.
\end{lem}

We are ready to give the proof of the main result of this note.

\begin{proof}
  [Proof of Theorem \ref{main_thm}]
  Let $\ee$ be a positive real and $k=\lfloor\log_2(1/\ee)\rfloor+1$. We set $t=t_{b^k-1}$. Let $K$ be a
   positive integer and $c:C^\uparrow_{\text{sur}}(b^\omega)\to K$ be a coloring of
   $C^\uparrow_{\text{sur}}(b^\omega)$. Let $\ell=b^k-1$ and $(s_i)_{i=0}^\ell$ the increasing enumeration of the set
    $b^k$ with respect to the lexicographical order on it. As we have already mentioned
    $(\mathcal{A}_b,\leq_\text{lex})$ is a countable dense unbounded linear order and therefore order
    isomorphic to $\mathbb{Q}$. For every $\mathbf{x}$ in $[\mathcal{A}_b]^\ell$, using Lemma
    \ref{finding_function}, we fix $f_\mathbf{x}$ in $C^\uparrow_{\text{sur}}(b^\omega)$ satisfying $(\max U^{f_\mathbf{x}}_{s_i})_{i<\ell}=\mathbf{x}$. We define a coloring $\widetilde{c}:[\mathcal{A}_b]^\ell\to K$, by
     setting $\widetilde{c}(\mathbf{x})=c(f_\mathbf{x})$. By Corollary \ref{Devlin_Cor} there exists a subset
      $Y$ of $\mathcal{A}_b$ order isomorphic to $\mathbb{Q}$ and $B\subseteq K$ of cardinality at most $t$ such
      that the image of $[Y]^\ell$ through $\widetilde{c}$ is equal to $B$. By Lemma \ref{lem_copies_Q_functions} we may pick $h$ in $C^\uparrow_{\text{sur}}(b^\omega)$ such that $Y_h=Y$. Then for every $f\in C^\uparrow_{\text{sur}}(b^\omega)$, setting $\mathbf{x}=(\max U^{f\circ h}_{s_i})_{i<\ell}$, by Lemma
      \ref{being_close} we have that $\rho_\infty(f\circ h, f_\mathbf{x})<\ee$ and by the definition of the coloring $\widetilde{c}$ and the choice of $B$ we get that $c(f_\mathbf{x})=\widetilde{c}(\mathbf{x})\in B$.
\end{proof}

\section{Necessity of the approximations}\label{Aprox_sec}

We recall some notation from \cite{KT} adapted in our setting.
  A subset $T$ of $2^{<\omega}$ is called \emph{subtree} if for every $s, t$ in $b^{<\omega}$ with $t\in T$ and $s$ initial segment of $t$ we have that $s\in T$. A node $s$ of a subtree $T$ is called a \emph{splitting node} of $T$ if there exist $t,t'$ in $T$ such that $s^\con(0)$ is initial segment of $t$ and $s^\con(1)$ is initial segment of $t'$, while by $Sp(T)$ we denote the set of all splitting nodes of $T$. A subtree $T$ is called \emph{perfect} if for every $s\in T$ there exists $t\in Sp(T)$ such that $s$ is proper initial segment of $t$. For every perfect subtree $T$ we denote by
   \[Bd(T)=\{x\in 2^\omega: x|n\in T\;\text{for all}\;n<\omega\}\]
   the \emph{body} of $T$.
   For every subset $A$ of $2^\omega$ we set
  \[A^{\uparrow}=\{t\in 2^{<\omega}:\;\text{there exists}\;y\in A\;\text{such that}\;t\text{ is initial segment of }y\}.\]
  Finally, a subset $A$ of $2^\omega$ is called \emph{non scattered} if it contains a subset order isomorphic to $\mathbb{Q}$. We recall the following result form \cite{KT}.

\begin{lem}
  \label{scattered_lem}
  Let $A\subseteq2^\omega$. If $A$ is non scattered then the set $T=\{s\in 2^{<\omega}:W_s\cap A\;\text{is non scattered}\}$ is a perfect subtree, where $W_s=\{x\in b^\omega:s$ is an initial segment of $x\}$.
\end{lem}

Although the following result is well known, we could not find a reference and we include its proof
for the convenience of the reader.

\begin{thm}
  \label{coloring_Q}
  There exists a coloring of $[\mathbb{Q}]^\eta$ into $\omega$ colors such that for every $Y\in[\mathbb{Q}]^\eta$, the set $[Y]^\eta$ witnesses all the colors.
\end{thm}
\begin{proof}
  Since $\mathcal{A}_2$ and $\mathbb{Q}$ are order isomorphic, it suffices to construct $c:[\mathcal{A}_2]^\eta\to\omega$ such that for every $Y\in [\mathcal{A}_2]^\eta$ we have that $[Y]^\eta$ witnesses all the colors.
  We define $c:[\mathcal{A}_2]^\eta\to\omega$ as follows.
  Let $Y\in [\mathcal{A}_2]^\eta$. We set
  \[T_Y=\{s\in 2^{<\omega}:W_s\cap Y\;\text{is non scattered}\}.\]
  By Lemma \ref{scattered_lem}
  we have that $T_Y$ is a perfect subtree. Since $T_Y$ is
  perfect we have that $Bd(T_Y)$ is a non-empty closed subset of $2^\omega$ (see \cite{K})
  and by Fact \ref{finding_max} it admits a maximum $m_Y$ and a minimum $\ell_Y$.
  Let $(t^Y_n)_{n<\omega}$ be the $\sqsubseteq$-increasing enumeration of the set
  $\{\ell_Y\}^\uparrow\cap Sp(T)$ and $(s^Y_n)_{n<\omega}$ be the $\sqsubseteq$-increasing enumeration of the set
  $\{m_Y\}^\uparrow\cap Sp(T)$. We set $c(Y)=\max\{i<\omega:|s^Y_i|<|t^Y_1|\}$.

  Let $Y\in[\mathcal{A}_2]^\eta$ and a color $r<\omega$. We will construct $Z\in[Y]^\eta$ such that $c(Z)=r$. We pick $n<\omega$ such that $m=\max\{i<\omega:|s^Y_i|<|t^Y_n|\}\meg r$. For notational simplicity we set $t_0=t^Y_n$ and $s_0=s^Y_{m-r+1}$. We define
  \[T=\{t\in T_Y:t\;\text{is}\;\sqsubseteq\text{-comparable with either}\;t_0\;\text{or}\;s_0\}.\]
  We set $I=\{y\in 2^\omega:\max_\text{lex} W_{t_0}<_\text{lex}y<_\text{lex}\min_\text{lex} W_{s_0}\}$, $Z'=Y\setminus I$ and $Z$ a $\subseteq$-maximal subset of $Z'$ order isomorphic to $\mathbb{Q}$. Clearly $Z\in[Y]^\eta$. It suffices to show that $T_Z=T$. It is easy to check that $|Z'\setminus Z|\mik1$. Thus  setting
  \[T_{Z'}=\{s\in 2^{<\omega}:W_s\cap Z'\;\text{is non scattered}\}\]
  we have that $T_Z=T_{Z'}$. By the definition of $T$ and $I$, for every $t\in T_Y\setminus T$  we have that
  $W_t\subseteq I$, while for every $t\in T$ we have that there exists $s_t\in T$ such that $t$ is initial segment of $s_t$ and $W_{s_t}\cap I=\emptyset$. Hence, for
  every $t\in T_Y\setminus T$ we have that $W_t\cap Z'=\emptyset$
  and therefore $t\not\in T_{Z'}$, while for every $t\in T$ we
  have that $W_t\cap Z'\supseteq W_{s_t}\cap Z'=W_{s_t}\cap Y$ and therefore $t\in T_{Z'}$.
  That is, $T_{Z'}\cap T_Y=T$. Since $Z'\subseteq Y$, we have
  that $T_{Z'}\subseteq T_Y$. Thus $T_{Z'}=T_Y$ and the proof is complete.
\end{proof}

\begin{cor}
  Let $b<\omega$ with $b\meg2$. There exists a coloring $c:C^\uparrow_{\text{sur}}(b^\omega)\to\omega$ such that for every $h\in C^\uparrow_{\text{sur}}(b^\omega)$ the set $\{f\circ h: f\in C^\uparrow_{\text{sur}}(b^\omega)\}$ witnesses all the colors.
\end{cor}
\begin{proof}
  Since $\mathcal{A}_b$ and $\mathbb{Q}$ are order isomorphic, by Theorem \ref{coloring_Q}
  we have that there exists a coloring $\widetilde{c}:[\mathcal{A}_b]^\eta\to \omega$ such that for every $Y\in[\mathcal{A}_b]^\eta$ the set $[Y]^\eta$ witnesses all the colors. We define $c:C^\uparrow_{\text{sur}}(b^\omega)\to\omega$ as follows. For every $f\in C^\uparrow_{\text{sur}}(b^\omega)$ we set $c(f)=\widetilde{c}(Y_f)$. Then for every $h\in C^\uparrow_{\text{sur}}(b^\omega)$ and $r<\omega$ there exists $Z\in [Y_h]^\eta$ such that $\widetilde{c}(Z)=r$ and by Lemma \ref{lem_copies_Q_functions} there exists $f\in C^\uparrow_{\text{sur}}(b^\omega)$ such that $Z=Y_{f\circ h}$ and therefore $c(f\circ h)=\widetilde{c}(Y_{f\circ h})=\widetilde{c}(Z)=r$.
\end{proof}

\section{Accuracy of the Ramsey degree}\label{deg_sec}
In this section we show that the Ramsey degree estimated in the Theroem \ref{main_thm}
is the best possible. In particular, we have the following result.

\begin{prop}
  \label{counterexample2}
  Let $b\in\omega$ with $b\meg2$. Also let $\ee>0$. Then there exists a coloring
  $c:C^\uparrow_{\text{sur}}(b^\omega)\to t(\ee)$ such that for every $B\subseteq t(\ee)$ such that
  there exists
  $h\in C^\uparrow_{\text{sur}}(b^\omega)$ satisfying that for every
  $f\in C^\uparrow_{\text{sur}}(b^\omega)$ there exists $g\in  C^\uparrow_{\text{sur}}(b^\omega)$
  with $\rho_\infty(f\circ h, g)<\ee$ and $c(g)\in B$, we have that $B=t(\ee)$.
\end{prop}

The above result is essential an application of the following result.

\begin{thm}
  [Devlin]
  \label{Devlin_counterexample}
  Let $\ell\in\omega$. Then there exists a coloring of $[\mathbb{Q}]^\ell$ into $t_\ell$ colors such that for
  every $Y\in [\mathbb{Q}]^\eta$, we have that $[Y]^\ell$ witnesses all the colors.
\end{thm}


\begin{lem}
  \label{A_b_tuples_close}
  Let $b\in\omega$ with $b\meg2$. Also let $\ell\in\omega$, $h\in C^\uparrow_{\text{sur}}(b^\omega)$ and
  $(x_i)_{i<\ell}\in [Y_h]^\ell$. We set
  $$U_0=\{x\in b^\omega:x\leq_\text{lex}x_0\}\text{ and }
  V_0=\{y\in b^\omega:y\leq_\text{lex}h(x_0)\},$$
  while for every $i=1,...,\ell-1$ we set
  $$U_i=\{x\in b^\omega:x_{i-1}<_\text{lex}x\leq_\text{lex}x_i\}\text{ and }
  V_i=\{y\in b^\omega:h(x_{i-1})<_\text{lex}y\leq_\text{lex}h(x_i)\}.$$
  Then $(h(x_i))_{i<\ell}\in[\mathcal{A}_b]^\ell$ and $(U_i)_{i<\ell}=(h^{-1}(V_i))_{i<\ell}$.
\end{lem}
\begin{proof}
  For every $i<\ell$ we pick $s_i\in b^{<\omega}$ such that $x_i=\max U^h_{s_i}=\max h^{-1}(W_{s_i})$.
  Moreover, we may assume that $s_0,\ldots,s_{\ell-1}$ are of the same length, by extending the shorter ones by $b-1$. Let $m$ be the common length of $s_0,\ldots,s_{\ell-1}$. Since $x_0,\ldots,x_{\ell-1}$ are distinct, we have that $s_0,\ldots,s_{\ell-1}$ are distinct too. Since $h$ is nondecreasing, we have that $s_0<_\text{lex}\ldots<_\text{lex}s_{\ell-1}$. We set $I_0=\{s\in b^m:s\leq_{\text{lex}}s_0\}$ and for every $1\mik i<\ell$ we set $I_i=\{s\in b^m:s_{i-1}<_\text{lex}s\leq_{\text{lex}}s_i\}$ . Then we have that
  $U_i=\cup_{s\in I_i}U^h_s=\cup_{s\in I_i}h^{-1}(W_s)=h^{-1}(\cup_{s\in I_i}W_s)$ for all $i<\ell$. Moreover, since $h$ is onto and increasing, we have that $h(x_i)=\max W_{s_i}$ for all $i<\ell$ and therefore
  $V_i=  \cup_{s\in I_i}W_s$ for all $i<\ell$. Hence $(U_i)_{i<\ell}=(h^{-1}(V_i))_{i<\ell}$ and $(h(x_i))_{i<\ell}=(\max V_i)_{i<\ell}=(\max W_{s_i})_{i<\ell}$ belongs to $[\mathcal{A}_b]^\ell$.
\end{proof}
We will also need the following strengthening of Lemma \ref{finding_function}.
\begin{lem}
  \label{finding_function_2}
  Let $b<\omega$ with $b\meg2$ and $k$ be a positive integer.
  Also let $h\in C^\uparrow_{\text{sur}}(b^\omega)$. Then
  $[Y_h]^{l_k}=\{(\max U^{f\circ h}_{s^k_i})_{i<l_k}: f \in C^\uparrow_{\text{sur}}(b^\omega)\}$.
\end{lem}
\begin{proof}
  Clearly $\{(\max U^{f\circ h}_{s^k_i})_{i<l_k}: f \in C^\uparrow_{\text{sur}}(b^\omega)\}\subseteq[Y_h]^{l_k}$. In order to prove the inverse inclusion,
  let $\mathbf{x}=(x_i)_{i<l_k}\in[Y_h]^{l_k}$. By Lemma \ref{A_b_tuples_close} we have that
  $(h(x_i))_{i<l_k}$ belongs to $[\mathcal{A}_b]^{l_k}$. By Lemma \ref{finding_function} there exists a map
  $f$ in $C^\uparrow_{\text{sur}}(b^\omega)$ such that $(\max U^f_{s_i^k})_{i< l_k}=(h(x_i))_{i<l_k}$.
  We set $U_0=\{x\in b^\omega:x\leq_\text{lex}x_0\}$ and $U_i=\{x\in b^\omega:x_{i-1}<_\text{lex}x\leq_\text{lex}x_i\}$ for all $1\mik i<l_k$.
  We also set $V_0=\{y\in b^\omega:y\leq_\text{lex}h(x_0)\}$ and $V_i=\{y\in b^\omega:h(x_{i-1})<_\text{lex}y\leq_\text{lex}h(x_i)\}$ for all $1\mik i<l_k$.
  Then $(U^f_{s^k_i})_{i<l_k}=(V_i)_{i<l_k}$.
  By Lemma \ref{A_b_tuples_close} we have that $(U_i)_{i<l_k}=(h^{-1}(V_i))_{i<l_k}$.
  Thus $(f\circ h)^{-1}(W_{s^k_i})=h^{-1}(f^{-1}(W_{s^k_i}))=h^{-1}(V_i)=U_i$ for all $i<l_k$.
  Hence $(\max U^{f\circ h}_{s^k_i})_{i<l_k}=\mathbf{x}$ and the proof is complete.
\end{proof}

\begin{proof}
  [Proof of Proposition \ref{counterexample2}]
  Let $k=\lfloor\log_2(1/\ee)\rfloor+1$ and $\ell=b^k-1$. Then $t_\ell=t(\ee)$. Since $\mathcal{A}_b$ and $\mathbb{Q}$ are order isomorphic, by Theorem, we have that  \ref{Devlin_counterexample} there exists a coloring
  $\widetilde{c}:[\mathcal{A}_b]^\ell\to t_\ell$ such that for every
  $Y\in[\mathcal{A}_b]^\eta$, the set $[Y]^\ell$ witnesses all the colors. Let
  $(s_i)_{i=0}^\ell$ the increasing enumeration of the set
  $b^k$ with respect to the lexicographical order on it.
  We define a coloring $c: C^\uparrow_{\text{sur}}(b^\omega)\to t_\ell$ by setting $c(f)=\widetilde{c}((\max U_{s_i}^f)_{i<\ell})$, for all $f$ in $C^\uparrow_{\text{sur}}(b^\omega)$.

  Let $B\subseteq t_\ell$ such that
  there exists a map
  $h\in C^\uparrow_{\text{sur}}(b^\omega)$ satisfying that for every
  $f\in C^\uparrow_{\text{sur}}(b^\omega)$ there exists $g\in  C^\uparrow_{\text{sur}}(b^\omega)$
  with $\rho_\infty(f\circ h, g)<\ee$ and $c(g)\in B$. We need to show that $B=t_\ell$.
  Indeed, let $r< t_\ell$. By the choice of $\widetilde{c}$ there exists
  $\mathbf{x}\in [Y_h]^\ell$ such that $\widetilde{c}(\mathbf{x})=r$. By Lemma \ref{finding_function_2}
  there exists $f \in C^\uparrow_{\text{sur}}(b^\omega)$ such that $(\max U^{f\circ h}_{s_i})_{i<\ell}=\mathbf{x}$ and therefore $c(f\circ h)=\widetilde{c}(\mathbf{x})=r$.
  By Lemma \ref{being_close}
  we have that for every $g\in C^\uparrow_{\text{sur}}(b^\omega)$ satisfying $\rho_\infty(f\circ h, g)<\ee$
  we have that $(\max U^{f\circ h}_{s_i})_{i<\ell}=(\max U^{g}_{s_i})_{i<\ell}$ and therefore $c(f\circ h)=c(g)=r$. Hence $r\in B$ and the proof is complete.
\end{proof}


\end{document}